\documentclass[12pt]{amsart}
\usepackage{amscd}
\usepackage{amssymb}
\usepackage{a4wide}
\usepackage{amstext}
\usepackage{amsthm}
\usepackage{xcolor}
\numberwithin{equation}{section}

\renewcommand{\phi}{\varphi}

\newcommand{\rl}{\mathbb{R}}

\newcommand{\E}{E_\Lambda}

\newtheorem{Thm}{Theorem}[section]
\newtheorem{theorem}[Thm]{Theorem}

\newtheorem{lemma}[Thm]{Lemma}
\newtheorem{proposition}[Thm]{Proposition}

\newtheorem{remark}[Thm]{Remark}
\newtheorem{definition}{Definition}

\begin{document}
\sloppy
\title[Gap Theorem for Separated Sequences without Pain]{Gap Theorem for Separated Sequences without Pain}
\author{Anton Baranov, Yurii Belov and Alexander Ulanovskii}
\smallskip
\address{
\newline \phantom{x}\,\, Anton Baranov,
\newline Department of Mathematics and Mechanics, St.~Petersburg State University, St.~Petersburg, Russia,
%\newline National Research University Higher School of Economics, St.~Petersburg, Russia,
\newline {\tt anton.d.baranov@gmail.com}
\smallskip
\newline \phantom{x}\,\, Yurii Belov,
\newline Chebyshev Laboratory, St.~Petersburg State University, St. Petersburg, Russia,
\newline {\tt j\_b\_juri\_belov@mail.ru}
\newline \phantom{x}\,\, Alexander Ulanovskii,
\newline  Stavanger University, 4036 Stavanger, Norway
\newline {\tt alexander.ulanovskii@uis.no}
}

\thanks{This work was supported by Russian Science Foundation grant 14-21-00035.}

\begin{abstract} We give a simple and straightforward proof of the Gap Theorem for separated sequences by A. Poltoratski and M. Mitkovski using the Beurling--Malliavin formula for the radius of completeness.
\end{abstract}


\maketitle

\section{Introduction and main result}

%Let $\Lambda$ be a real sequence. Famous {\it upper Beurling--Malliavin density} $D^{BM}(\Lambda)$ (which will be defined %below) gives us  an answer about effective calculation of {\it radius of completeness $R(\Lambda)$} of $\Lambda$.

For a real discrete set $\Lambda$ consider the system of exponentials $$\E:=\{e^{i\lambda t}\}_{\lambda\in\Lambda}.$$ The famous Beurling--Malliavin theorem gives an effective formula for the {\it completeness radius} $R_\Lambda$ of $\E$ in terms of the so-called {\it upper Beurling--Malliavin density} $D^{BM}(\Lambda)$ (to be defined below). More precisely, put
$$R(\Lambda)=\sup\{a: \E\text{ is complete in } L^2(-a,a)\}.$$
Then the Beurling--Malliavin theorem \cite{bm} (for detailed exposition see \cite{hj, Koo}) states

\begin{theorem} $R(\Lambda) = \pi D^{BM}(\Lambda).$\end{theorem}

The elegance and finality of this result impresses mathematicians over 50 years. Nevertheless, the dual concept of {\it the lower Beurling--Malliavin density} $D_{BM}(\Lambda)$ had found practical use only some years ago. 

 Let $\Lambda$ be a  separated set, i.e. 
\begin{equation}\label{sep}
d(\Lambda):=\inf_{\lambda,\lambda'\in\Lambda,\lambda\ne\lambda'}|\lambda-\lambda'|>0.
\end{equation} Denote by $M(\Lambda)$ the set of finite complex measures supported by $\Lambda$. 
The  {\it gap characteristic}  $G(\Lambda)$ is defined by
$$G(\Lambda)=\sup\{a: \exists  \mu\in M(\Lambda)\setminus\{0\} \text{ such that }  \hat{\mu}(x)=0, x\in(-a,a)\}.$$
%It is easy to see that in this definition we can consider only measures with finite variation. The quantity $G(\Lambda)$ is called {\it the gap characteristic of $\Lambda$}.

In 2010 M. Mitkovski and A. Poltoratski \cite{mp} proved the following result:
\begin{theorem}  Assume $\Lambda\subset\mathbb{R}$ is  a separated set. Then
$$G(\Lambda)=\pi D_{BM}(\Lambda).$$
\label{mainth}
\end{theorem}
The proof of this result in \cite{mp} uses theory of model subspaces of Hardy class $H^2$, theory of Toeplitz kernels and some other tools.

The aim of our paper is to show that Theorem \ref{mainth} can be {\it directly derived } from Theorem~1.1. So, instead of two difficult results in harmonic analysis essentially we have only one.

It should be noted that for non-separated sequences $\Lambda$ the formula for gap characteristic was recently found by A. Poltoratski \cite{p}. This formula is much more involved and includes {\it the concept of energy}.  It is not clear (at least to the authors) whether this formula also can be directly derived from the classical Beurling--Malliavin theory.

\section{Proof of Theorem \ref{mainth}}
 Theorem \ref{mainth} is a straightforward consequence of Theorem 1.1 and three rather elementary results stated below.

 The first result shows that the upper and lower Beurling--Malliavin densities are in a sense complementary:

\begin{proposition} Assume $\Lambda\subset \alpha\mathbb{Z}, \alpha>0$. Then
$$D_{BM}(\Lambda)+D^{BM}(\alpha\mathbb{Z}\setminus \Lambda)=1/\alpha.$$
\label{p1}
\end{proposition}

Here and below we put $\alpha\mathbb{Z}=\{\alpha n:n\in\mathbb{Z} \}$.

A similar result is true for the  completeness radius and the gap characteristic:

\begin{proposition} Assume $\Lambda\subset \alpha\mathbb{Z}, \alpha>0$. Then
$$G(\Lambda)+R(\alpha\mathbb{Z}\setminus\Lambda)=\pi/\alpha.$$
\label{p2}
\end{proposition}

Given a separated set $\Lambda$, we consider its perturbations: 
\begin{equation}\label{per}
\tilde{\Lambda}=\{\lambda+\varepsilon_\lambda:  \lambda\in\Lambda\}.
\end{equation} 
%We will assume  that all perturbations satisfy $|\varepsilon_\lambda|<d(\Lambda)/4,$ where $d(\Lambda)$ is the separation constant in (\ref{sep}). Then, clearly, $\tilde{\Lambda}$ is also a separated set.
The third result shows that  some {\it positive perturbations} do not change the gap characteristic:

\begin{proposition} Assume $\Lambda$ is a separated set. For every  positive number $\delta<d(\Lambda)/4,$ where $d(\Lambda)$ is the separation constant in \eqref{sep}, and  all numbers $\varepsilon_\lambda$ satisfying 
\begin{equation}\label{eps} \delta/2<\varepsilon_\lambda<\delta,\quad\lambda\in\Lambda,
\end{equation} 
the set $\tilde{\Lambda}$ in \eqref{per}  satisfies
$$
G(\tilde{\Lambda})=G(\Lambda).
$$
\label{p3}
\end{proposition}

Observe, that  condition $\delta<d(\Lambda)/4$  implies that $\tilde{\Lambda}$ itself is a separated set.

We postpone the proofs of Propositions \ref{p1}-\ref{p3}. Now, let us prove Theorem \ref{mainth}. 

\begin{proof} We consider two cases.

(i) Assume additionally that $\Lambda$ is a subset of $\alpha\mathbb{Z}$, for some $\alpha>0$. 
In view of Theorem~1.1 and Propositions 2.1 and 2.2, we have 
$$
G(\Lambda)=\pi/\alpha-R(\alpha\mathbb{Z}\setminus\Lambda)=\pi/\alpha-\pi D^{BM}(\alpha\mathbb{Z}\setminus \Lambda)$$$$=\pi/\alpha-\pi(1/\alpha-D_{BM}(\Lambda))=\pi D_{BM}(\Lambda).
$$
 Theorem \ref{mainth} is proved for the subsequences of $\alpha\mathbb{Z}$.

(ii) Fix any separated set $\Lambda$ and positive $\delta<d(\Lambda)/4$. Clearly, there is  a set $\tilde{\Lambda}$ \eqref{per} satisfying \eqref{eps} and such that $\tilde{\Lambda}\subset\alpha\mathbb{Z}$, for some sufficiently small $\alpha>0.$ By Proposition~\ref{p3}, $G(\tilde{\Lambda})=G(\Lambda)$.

 Using the definition of  lower Beurling--Malliavin density (see below), one may easily check that  $D_{BM}(\tilde{\Lambda})= D_{BM}(\Lambda)$. So,  by (i), we  conclude that  $$G(\Lambda)=G(\tilde{\Lambda})=\pi D_{BM}(\tilde{\Lambda})=\pi D_{BM}(\Lambda).$$

\end{proof}

So, we have used Theorem 1.1 for {\it separated sets} to deduce Theorem \ref{mainth}. 
We notice, that in fact these two results are equivalent. The converse implication is given by

\begin{remark}
 Beurling--Malliavin's Theorem 1.1 for separated sets follows from Theorem~\ref{mainth}.
\end{remark}

To check this, one may use a similar proof where instead of  Propositions  2.3 one needs 

\begin{proposition} Assume $\Lambda$ is a separated set. There exists $\delta>0$ such that for all numbers $|\varepsilon_\lambda|<\delta,\lambda\in\Lambda,$ the set $\tilde{\Lambda}$ in \eqref{per} satisfies
$$
R(\tilde{\Lambda})=R(\Lambda).
$$
\label{p4}
\end{proposition}

Clearly, this result easily follows from Theorem 1.1 and the definition of $D^{BM}$. We remark that one may prove it by elementary means involving standard estimates of Weierstrass products.

\section{Proof of Proposition \ref{p1}}
There exist at least five definitions of the upper Beurling--Malliavin density (see paper \cite{kt} which is devoted to equivalence of different definitions). We start with the most well-known:
\begin{definition} We will say that the sequence $\Lambda\subset\mathbb{R}$ is strongly $a$-regular if its counting function $n_\Lambda$ satisfies
$$\int_\mathbb{R}\frac{|n_\Lambda(x)-ax|}{1+x^2}dx<\infty.$$
\end{definition}
\begin{definition} The upper Beurling--Malliavin density $D^{BM}(\Lambda)$ is the infimum of numbers $a$ such that the function $n_{\Lambda\cup\Lambda'}$ is strongly $a$-regular for some $\Lambda'\subset\mathbb{R}$.
\label{def1}
\end{definition}
This definition goes back to J.-P.Kahane. The original definition given by Beurling and Malliavin used the notion of {\it short system of intervals}, see \cite[p. 397--398]{kt}. We need one more equivalent definition which was found by R. Redheffer, see \cite{rr, rr2}.
\begin{definition} The upper Beurling--Malliavin density $D^{BM}(\Lambda)$ is the infimum of numbers $a$ such that there exists
a sequence of distinct integers $n_k$ such that
$$\sum_{\lambda_k\in\Lambda}\biggl{|}\frac{1}{\lambda_k}-\frac{a}{n_k}\biggr{|}<\infty.$$
\label{def2}
\end{definition}
Now we give a "dual" $ $ definition of {\it the lower Beurling--Malliavin density}.
\begin{definition} The lower  Beurling--Malliavin density $D_{BM}(\Lambda)$ is the supremum of numbers $a$ such that the function $n_{\Lambda'}$ is strongly $a$-regular for some $\Lambda'\subset\Lambda$.
\label{def3}
\end{definition}

 From the equivalence of Definitions \ref{def1} and \ref{def2} it follows that if $D^{BM}(\Lambda)=a$, then
for every $b> a$ there exists $\Lambda_0\subset b^{-1}\mathbb{Z}$ such that $n_\Lambda-n_{\Lambda_0}\in L^1((1+x^2)^{-1}dx)$. Hence, for every $b>a$ the sequence $\Lambda'$ in Definition \ref{def1}
can be taken as a subset of the arithmetic progression $b^{-1}\mathbb{Z}$.
%From this it is easy to see that
%if $\Lambda\subset{\mathbb{Z}}$ then $\Lambda'$ can be taken as a subset of $\mathbb{Z}\setminus{\Lambda}.$

Let us now prove Proposition \ref{p1}. For simplicity, using re-scaling, we may assume that $\alpha=1$.

\begin{proof} Set $\Gamma:=\mathbb{Z}\setminus\Lambda$. First of all we will show that if $D^{BM}(\Gamma)=a$, then for any $b>a$ we can choose $\Gamma'\subset\Lambda$ such that $n_{\Gamma\cup\Gamma'}$ is strongly $b$-regular. Indeed, let as above
 $\Gamma_0=\{b^{-1}n_k\}\subset b^{-1}\mathbb{Z}$ (where $n_k$ are distinct integers as in Definition \ref{def2}) and $n_\Gamma-n_{\Gamma_0}\in L^1((1+x^2)^{-1}dx)$. Put
$\Gamma_1=b^{-1}\mathbb{Z}\setminus\Gamma_0$. We have that $\Gamma\cup\Gamma_1$ is strongly $b$-regular.
It would be natural to put
$\Gamma' = \{[\gamma]:\gamma\in\Gamma_1\}.$ However with this definition it is possible that $\Gamma'\cap\Gamma\neq\emptyset$. To avoid this
we define
$$\Gamma'_{ex}=\{\gamma_k\in\Gamma: \gamma_k\in[\Gamma_1] \}$$
and shift the points from $\Gamma'_{ex}$ in the following way:
$$\Gamma'=([\Gamma_1]\setminus\Gamma'_{ex})\cup\{[b^{-1}n_k]: \gamma_k\in\Gamma'_{ex}\}.$$
Using  again the fact that $n_\Gamma-n_{\Gamma_0}\in L^1((1+x^2)^{-1}dx)$ and that $[b^{-1}n_k]\not\in[\Gamma_1]$, $\gamma_k\in\Gamma'_{ex}$ we get that $n_{\Gamma\cup\Gamma'}$ is
strongly $b$-regular.

\medskip

Now suppose that $n_{\Gamma\cup\Gamma'}$ is strongly $a$-regular for some $\Gamma'\subset\Lambda$. Then $n_{\mathbb{Z}\setminus{(\Gamma\cup\Gamma')}}$ is strongly $(1-a)$-regular. Since $\mathbb{Z}\setminus(\Gamma\cup\Gamma')\subset\Lambda$, we have $D_{BM}(\Lambda)\geq 1-a$ whence $D^{BM}(\Gamma)+D_{BM}(\Lambda)\geq 1$. 

On the other hand, if $n_{\Lambda''}$ is strongly $(1-a)$-regular for some $\Lambda''\subset\Lambda$, then $n_{\mathbb{Z}\setminus{\Lambda''}}$ is strongly $a$-regular and $\Gamma\subset\mathbb{Z}\setminus{\Lambda''}$. 
 So, $D^{BM}(\Gamma)+D_{BM}(\Lambda)\leq 1$.
\end{proof}

\section{Proof of Proposition \ref{p2}}

\begin{proof} Again, we may assume that $\alpha=1$ and put $\Gamma:=\mathbb{Z}\setminus\Lambda$. It is clear that $R(\Gamma),G(\Lambda)\leq 2\pi$. 

If the system $E_\Gamma:=\{e^{i\gamma t}\}_{\gamma\in\Gamma}$ is not complete in $L^2(0,2a)$, $0<a<\pi$,
then there exists a non-trivial function $f\in L^2(\rl)$ which vanishes outside  $(0,2a)$ and $f\perp E_\Gamma$. Take any small positive number $\epsilon$ and consider the convolution $g=f\ast h$, where $h$ is a smooth function supported by $[0,\varepsilon]$. Then $g$ is smooth, vanishes outside $(0,2a+\varepsilon)$ and is orthogonal to $E_\Gamma$. Since $\{e^{int}\}_{n\in\mathbb{Z}}$ is an orthogonal basis in $L^2(0,2\pi)$ we obtain
$$g(x)=\sum_{n\in\mathbb{Z}}a_ne^{inx}=\sum_{n\in\mathbb{Z}\setminus{\Lambda}}a_ne^{inx},\quad \{a_n\}\in\ell^1.$$
So, the measure  $$\mu:=\sum_{n\in \Gamma}a_n\delta_n$$ belongs to $M(\Lambda)$ and  has a spectral gap of length at least $2\pi - 2a-\varepsilon$. Since $\varepsilon$ is arbitrary, we conclude that $R(\Gamma)+G(\Lambda)\geq \pi$.

Now, suppose that there exists a non trivial measure $$\mu:=\sum_{n\in\Lambda}a_n\delta_n\in M(\Lambda)$$ with a spectral gap of size $2a$. Without loss of generality we can assume that $\hat{\mu}\equiv 0$ on $(0,2a)$.
Put $g(x)=\hat{\mu}\bigl{|}_{(0,2\pi)}$. We have $g\in L^2(0,2\pi)$ and $g\perp E_\Gamma$. Hence,
$R(\Gamma)\leq\pi-a$. So, $R(\Gamma)+G(\Lambda)\leq \pi$ and Proposition \ref{p2} is proved.
\end{proof}

%\medskip

\section{Proof of Proposition \ref{p3}}

We will use the following well-known fact (see e.g. \cite[Lemma 2]{mp}).  For the sake of completeness we give its proof here.
\begin{lemma}
Let $\mu\in M(\rl)$. Then the Fourier transform of $\mu$ vanishes on $[-a,a]$ if and only if
\begin{equation}
\lim_{y\rightarrow\pm\infty}e^{by}\int_{\mathbb{R}}\frac{d\mu(t)}{iy-t}=0,
\label{f1}
\end{equation}
for every $b\in(-a,a)$.
\end{lemma}
\begin{proof}
Let $\mu$ be such that $\int_\mathbb{R}e^{ibt}d\mu(t)=0$, $|b|\leq a$. Then, for any $z\in\mathbb{C}$,
$$\int_\mathbb{R}\frac{e^{ibt}-e^{ibz}}{t-z}d\mu(t)=ie^{ibz}\int_\mathbb{R}\int_0^be^{iu(t-z)}du\,d\mu(t)=0.$$
Hence,
\begin{equation}
\lim_{y\rightarrow\pm\infty}iye^{by}\int_{\mathbb{R}}\frac{d\mu(t)}{iy-t}=
\lim_{y\rightarrow\pm\infty}iy\int_\mathbb{R}\frac{e^{ibt}}{iy-t}d\mu(t)=\int_\mathbb{R}e^{ibt}d\mu(t)=0.
\label{f2}
\end{equation}
Conversely, for any $b\in(-a,a)$ put
$$H(z):=\int_\mathbb{R}\frac{e^{ibt}-e^{ibz}}{t-z}d\mu(t).$$
Clearly $H$ is an entire function of Cartwright class (which means that its {\it logarithmic integral} converges, see \cite{Levin}, Lec.16). On the other hand, by \eqref{f2} we have $\lim_{|y|\rightarrow\infty}|H(iy)|=0$. Hence,
$H(iy)\equiv0$ and the statement follows from \eqref{f2}.
\end{proof}

We will also need an elementary lemma:

\begin{lemma} Let $\Lambda$ be a separated set. Then

$(i)$ $G(\Lambda)=G(\Lambda-x)$, for  every  $x\in\mathbb{R}$, where $\Lambda-x:=\{\lambda-x:\lambda\in\Lambda\}$;

$(ii)$  $G(\Lambda\cup\{\lambda'\})=G(\Lambda)$, for  every  $\lambda'\not\in\Lambda$;

$(iii)$  if $G(\Lambda)>0$, then for every positive $a<G(\Lambda)$ there is a measure $$\mu=\sum_{\lambda\in\Lambda} c_\lambda\delta_\lambda$$with spectral gap $[-a,a]$ and such that
$$
|c_\lambda|=O(|\lambda|^{-2}), \qquad |\lambda|\to\infty.
$$

\end{lemma}

Let us, for example, check (iii). Take a positive $\varepsilon$ satisfying $a+\varepsilon<G(\Lambda)$, and choose any measure $\nu$ with spectral gap on  $[-a-\epsilon,a+\epsilon].$  Then put $\mu=h\nu$, where $h$ is a fast decreasing  function whose spectrum lies on   $[-\epsilon,\epsilon].$

Now, we prove Proposition 2.3. 

%We will use the following notations: Given positive quantities $U(x)$, $V(x)$, the notation $U(x)\lesssim V(x)$ (or, equivalently, $V(x)\gtrsim U(x)$) means that there is a constant $C$ such that $U(x)\leq CV(x)$ holds for all $x$ in the set in question. We write $U(x)\simeq V(x)$ if both $U(x)\lesssim V(x)$ and $V(x)\lesssim U(x)$.

\begin{proof} In the proof below we will assume that  $G(\Lambda)>0$, and show that $G(\tilde{\Lambda})\geq G(\Lambda)$ for every $\tilde{\Lambda}$ satisfying the assumptions of Proposition 2.3.
The same proof works as well in the opposite direction: If $G(\tilde{\Lambda})>0$ then $G(\Lambda)\geq G(\tilde{\Lambda})$. It will follow that $G(\tilde{\Lambda})= G(\Lambda)$.
It also shows that $G(\tilde{\Lambda})= 0$ if $G(\Lambda)=0$.

The proof will consist of several steps.

1.   We may write
$$
\Lambda=\{\lambda_j:j\in\mathbb{Z}\},\qquad \tilde{\Lambda}=\{\tilde{\lambda_j}:j\in\mathbb{Z}\},
$$where
$$
\delta/2<\tilde{\lambda_j}-\lambda_j<\delta, \qquad j\in\mathbb{Z}.
$$We may also assume that $0\not\in \tilde{\Lambda}\cup\Lambda$.

2. Consider the meromorphic function
$$
\varphi(z):=-\prod_{j\in\mathbb{Z}}\frac{1-z/\lambda_j}{1-z/\tilde{\lambda_j}}.
$$
One may check that the product converges (see, for example, \cite{Levin}, p. 220).

Since
$$
\arg\frac{1-z/\lambda_j}{1-z/\tilde{\lambda_j}}=\arg(z-\lambda_j)-\arg(z-\tilde{\lambda_j}),
$$
and since $\Lambda$ and $\tilde{\Lambda}$ are interlacing, one may see that $\Im\varphi(z)>0$ whenever  $\Im z>0$. Hence (see \cite{Levin}, p. 220, 221), $\varphi$ admits a representation
$$\varphi(z)=b_1z+b_2+\sum_{\tilde{\lambda}_k\in\tilde{\Lambda}} c_k\biggl{(}\frac{1}{\tilde{\lambda}_k-z}-\frac{1}{\tilde{\lambda}_k}\biggr{)},$$where
$ b_1\geq0, c_k>0,b_2\in\mathbb{R}$ and 
\begin{equation}\label{ck}
\sum_k\frac{c_k}{\tilde{\lambda}^2_k}<\infty.
\end{equation}

Clearly, we have
\begin{equation}\label{y}
|\varphi(iy)|=O(|y|), \qquad |y|\to\infty.
\end{equation}

3. Fix a positive number $a<G(\Lambda)$, and take a measure
$$
\mu=\sum_{j\in\mathbb{Z}}d_j\delta_{\lambda_j}
$$ which has a spectral gap on $[-a,a]$ and whose coefficients satisfy
\begin{equation}\label{dj}
|d_j|=o(|j|^{-2}), \quad |j|\to\infty.
\end{equation}

Then fix two points $x_1,x_2\not\in\Lambda\cup\tilde{\Lambda}$ and consider the meromorphic function
$$
\psi(z):=\frac{\varphi(z)}{(z-x_1)(z-x_2)}\sum_{j\in\mathbb{Z}}\frac{d_j}{z-\lambda_{j}}.
$$
It is easy to check  that 
$$
\psi(z)=\sum_{k\in\mathbb{Z}}\frac{e_k}{z-\tilde{\lambda}_{k}}+\sum_{j=1}^2\frac{f_j}{z-x_j},
$$where 
$$
e_k=\lim_{z\to\tilde{\lambda}_k}(z-\tilde{\lambda}_k)\psi(z)=\frac{c_k}{(\tilde{\lambda}-x_1)(\tilde{\lambda}-x_2)}\sum_{j\in\mathbb{Z}}\frac{d_j}{\lambda_j-\tilde{\lambda_k}}.
$$Since $|\lambda_j-\tilde{\lambda_k}|>\delta/2$, by (\ref{ck}) and (\ref{dj}) we see that $\{e_k:k\in\mathbb{Z}\}\in l^1.$

4. By Lemma 5.1,  we have 
$$
e^{b|y|}\sum_{j\in\mathbb{Z}}\frac{d_j}{iy-\lambda_{j}}\to 0, \quad |y|\to\infty, \mbox{ for every } 0<b<a.
$$
So, by \eqref{y}, the same estimate holds for the function
$$
\sum_{k\in\mathbb{Z}}\frac{e_k}{iy-\tilde{\lambda}_{k}}+ \sum_{j=1}^2\frac{f_j}{iy-x_j}.
$$This shows that $G(\tilde{\Lambda}\cup\{x_1,x_2\})\geq a.$
Since this is true for every $a<G(\Lambda)$, by Lemma 5.2 (ii), we conclude that  $G(\tilde{\Lambda})\geq G(\Lambda)$. 
\end{proof}

\end{document}